\documentclass[11pt]{amsart}
\usepackage{amsmath,amsthm,amsfonts,amssymb,mathrsfs}
\usepackage[margin=1in]{geometry}
\usepackage[bookmarks]{hyperref}
\usepackage{verbatim}
\usepackage{todonotes}
\usepackage{fancyvrb }
\usepackage{color}

\newenvironment{enumalph}{\begin{enumerate}  }{\end{enumerate}}

\newcommand{\set}[1]{\left\{ #1 \right\}}

\newtheorem{theorem}{Theorem}[subsection]

\newtheorem{proposition}[theorem]{Proposition}

\newtheorem{conjecture}[theorem]{Conjecture}

\newtheorem{corollary}[theorem]{Corollary}

\newtheorem{lemma}[theorem]{Lemma}
\input xy
\xyoption{all}

\theoremstyle{definition}

\newtheorem{remark}[theorem]{Remark}

\DeclareMathOperator{\im}{Im}

\DeclareMathOperator{\ch}{ch}

\DeclareMathOperator{\Dist}{Dist}

\DeclareMathOperator{\opH}{H}

\DeclareMathOperator{\ind}{ind}
\DeclareMathOperator{\Lie}{Lie}

\DeclareMathOperator{\rad}{Nilrad}

\newcommand{\calO}{\mathcal{O}}

\newcommand{\calF}{\mathcal{F}}

\newcommand{\spec}{\mbox{{\rm{Spec}}\;}}
\newcommand{\red}{\rm{red}}

\newcommand{\frakb}{\mathfrak{b}}
\newcommand{\g}{\mathfrak{g}}

\newcommand{\fraku}{\mathfrak{u}}

\newcommand{\fraksl}{\mathfrak{sl}}
\newcommand{\Hbul}{\opH^\bullet}                                  

\newcommand{\N}{\mathcal{N}}

\newcommand{\Fp}{\mathbb{F}_p}

\newcommand{\E}{\mathbb{E}}

\newcommand{\R}{\mathbb{R}}
\newcommand{\U}{\mathbb{U}}
\newcommand{\Z}{\mathbb{Z}}

\newcommand{\Ua}{U_\upA}

\newcommand{\Uq}{\U_q}

\newcommand{\Uz}{U_\zeta}

\newcommand{\uz}{u_\zeta}

\newcommand{\upA}{\mathsf{A}}

\begin{document}

\title[Cohomology for Frobenius kernels of $SL_2$]{Cohomology for Frobenius kernels of $SL_2$}

\author{Nham V. Ngo}
\address{Department of Mathematics, Statistics, and Computer Science\\ University of Wisconsin-Stout \\ Menomonie \\ WI~54751, USA}
\email{ngon@uwstout.edu}

\address{{\bf Current address:} Department of Mathematics and Statistics \\ Lancaster University \\ Lancaster \\ LA1 4YW, UK}
\email{n.ngo@lancaster.ac.uk}

\date{\today}

\maketitle

\begin{abstract}
Let $(SL_2)_r$ be the $r$-th Frobenius kernels of the group scheme $SL_2$ defined over an algebraically field of characteristic $p>0$. In this paper we give for $r\ge 1$ a complete description of the cohomology groups for $(SL_2)_r$. We also prove that the reduced cohomology ring $\opH^\bullet((SL_2)_r,k)_{\red}$ is Cohen-Macaulay. Geometrically, we show for each $r\ge 1$ that the maximal ideal spectrum of the cohomology ring for $(SL_2)_r$ is homeomorphic to the fiber product $G\times_B\fraku^r$. Finally, we adapt our calculations to obtain analogous results for the cohomology of higher Frobenius-Luzstig kernels of quantized enveloping algebras of type $SL_2$.    
\end{abstract}

\section{Introduction}
\subsection{} In recent years, the cohomology and representation theory of Frobenius kernels has received considerable interest. It is well-known that the category of restricted representations for the Lie algebra of an algebraic group in characteristic $p>0$ is equivalent to that of representations over the first Frobenius kernels of the algebraic group. This connection has inspired many investigations of the cohomology for the first Frobenius kernels of algebraic groups. For higher Frobenius kernels, the cohomology is of interest as it provides information about the original group scheme. For instance, let $G$ be a simple, simply-connected algebraic group defined over an algebraically closed field $k$ of characteristic $p>0$. Denote by $G_r$ the $r$-th Frobenius kernel of $G$. Then for each $i\ge 0$, we can identify $\opH^i(G,M)$ with the inverse limit $\varprojlim_{r}\opH^i(G_r,M)$ \cite[II.4.12]{Jan:2003}. In 2006, Bendel, Nakano, and Pillen computed the first and second degree cohomology of the $r$-th Frobenius kernels of $G$ \cite{BNP:2004},\cite{BNP:2006}. However, the state of affairs for higher degree cohomology remains an open problem in general.

Via geometry, more is known about the cohomology of Frobenius kernels. It was first noticed for semisimple algebraic groups by Friedlander and Parshall that there is an isomorphism between the cohomology ring for $G_1$ and the coordinate ring of the nilpotent cone of the Lie algebra $\g=$Lie($G$) \cite{FP:1986}. A generalization was obtained for the higher Frobenius kernels by Suslin, Friedlander, and Bendel \cite{SFB1:1997}, \cite{SFB2:1997}. They constructed a homeomorphism between the spectrum of the $G_r$-cohomology ring and the variety of $r$-tuples of commuting nilpotent elements in $\g$. Moreover, in the case $G=SL_2$, Suslin, Friedlander, and Bendel explicitly computed the support varieties for the induced module and simple module of every dominant weight \cite[Proposition 6.10]{SFB2:1997}. 

On the subject of cohomology for finite groups, Benson and Carlson studied the algebraic structure of cohomology rings \cite{Car:1994}. One of their major concerns is whether the cohomology ring of a finite group is Cohen-Macaulay. It is known that if $E$ is an elementary abelian $p$-group, then the cohomology ring for $E$ is Cohen-Macaulay \cite[5.18]{Ben:1998}. In this paper we are concerned with the same question applied to the cohomology of Frobenius kernels. Since the cohomology ring for the first Frobenius kernel of G is isomorphic to the coordinate ring of the nullcone, it is Cohen-Macaulay. Our proof in Section \ref{Cohen-Macaulayrings} establishes the Cohen-Macaulayness of the cohomology rings for the higher Frobenius kernels of $SL_2$.  

\subsection{Main results} The paper is organized as follows. We establish basic notation in Section \ref{notation}. Then in Section \ref{generalsection}, we derive explicit spectral sequences to compute the cohomology of $B_r$ from the spectral sequences stated in \cite[I.9.14]{Jan:2003} and recall a strategy to calculate the cohomology of $G_r$.

Section \ref{cohomology} contains cohomology calculations for the subgroup schemes $U_r, B_r$, and $G_r$ of $SL_2$ where $B$ is a fixed Borel subgroup of $G$ and $U$ is its unipotent radical subgroup. The computation for cohomology of $G_r$ provides a new proof of a result of van de Kallen concerning good filtrations on the cohomology groups $\opH^n((SL_2)_r,\opH^0(\lambda))$ with $n\ge 0$ and an arbitrary dominant weight $\lambda$ \cite{vdK:1993}. One will notice that the results in this section depend highly on the multiplicity of certain dominant weights. Hence, in the next section we introduce an algorithm that calculates the character multiplicity of a weight $\mu$ in $\Hbul(B_r,\lambda)$, and hence the multiplicity of $\opH^0(\mu)$ in $\Hbul(G_r,\opH^0(\lambda))$. In Appendix, we provide computer calculations showing that our algorithm is faster than the one encoded by using Ehrhart polynomial. We compute the rings $\Hbul(B_r,k)_{\red}$ and $\Hbul(G_r,k)_{\red}$ in Section \ref{reduced cohomology} in order to investigate their geometric structures. In the process, we develop several techniques to study reduced commutative rings. This computation shows that there is a homeomorphism between the spectra of the reduced $G_r$-cohomology ring and the ring of global sections on $G\times^B\fraku^r$ (cf. Proposition \ref{G_rscheme}). This also plays a key role in showing that the reduced $B_r$- and $G_r$-cohomology rings are Cohen-Macaulay (Section \ref{Cohen-Macaulayrings}). This new result inspires us to state a conjecture on the Cohen-Macaulayness of the cohomology ring for the Frobenius kernels of an algebraic group. 

The last section is devoted to studying analogous results for the higher Frobenius-Lusztig kernels of quantum groups. These objects, which are analogs for the hyperalgebras of the higher Frobenius kernels of an algebraic group, were first defined by Drupieski in \cite{Dru:2011} in the context of quantized enveloping algebras defined over fields of positive characteristic. He verified various properties of the Frobenius-Lusztig kernels which are similar to the results obtained in the case of characteristic 0.     

\section{Notation}\label{notation}
\subsection{Representation theory} Let $k$ be an algebraically closed field of characteristic $p > 0$. Let $G$ be a simple, simply-connected algebraic group over $k$, defined and split over the prime field $\Fp$. Denote by $h$ the Coxeter number of $G$. Fix a maximal torus $T$ of $G$, and denote by $\Phi$ the root system of $T$ in $G$. Fix a set $\Pi = \{ \alpha_1,\ldots,\alpha_n \}$ of simple roots in $\Phi$, and let $\Phi^+$ be the corresponding set of positive roots. Let $B \subset G$ be the Borel subgroup of $G$ containing $T$ and corresponding to $\Phi^+$, the set of positive roots, and let $U \subset B$ be the unipotent radical of $B$. 

Let $W$ be the Weyl group of $\Phi$; it is generated by the set of simple reflections $\set{s_\alpha: \alpha \in \Pi}$. Write $\ell: W \rightarrow \mathbb{N}$ for the standard length function on $W$, and let $w_0 \in W$ be the longest element. Let $(\cdot,\cdot)$ be the inner product associated with the Euclidean space $\E := \Z \Phi \otimes_\Z \R$. Given $\alpha \in \Phi$, let $\alpha^\vee := 2\alpha/(\alpha,\alpha)$ be the corresponding coroot. Set $\alpha_0$ to be the highest short root of $\Phi$, and $\rho$ to be one-half the sum of all positive roots in $\Phi$. Then the Coxeter number of $\Phi$ is $h=(\rho,\alpha^\vee_0)+1$. Suppose $\lambda=\sum_{\alpha\in\Pi}m_{\alpha}\alpha$ a weight in $X$, then the height of $\lambda$ is defined as ht$(\lambda)=\sum_{\alpha\in\Pi}m_{\alpha}$.

Let $X$ be the weight lattice of $\Phi$, defined by the $\Z$-span of the fundamental weights $\{ \omega_1,\ldots,\omega_n \}$, and let $X^+ \subset X$ be the set of dominant weights. Simple $G$-modules are indexed by $\lambda\in X^+$, and denoted by $L(\lambda)$. The simple module $L(\lambda)$ can be identified with the socle of the induced module $H^0(\lambda)=\ind_{B^-}^{G}\lambda$. Set $\g = \Lie(G)$, the Lie algebra of $G$, $\frakb = \Lie(B)$, $\fraku = \Lie(U)$. Denote by $S^\bullet(\fraku^*)$ and $\Lambda^\bullet(\fraku^*)$ respectively the symmetric algebra and exterior algebra over $\fraku^*$. Throughout this paper, the symbol $\otimes$ means the tensor product over the field $k$, unless otherwise stated. Suppose $H$ is an algebraic group over $k$ and $M$ is a (rational) module of $H$. Denote by $M^H$ the submodule consisting of all the fixed points of $M$ under the $H$-action.

Let $Y$ be a group scheme. Then for every positive integer $r$, the scheme $Y^{(r)}$ is defined by for each $k$-algebra $A$,
\[ Y^{(r)}(A)=Y(A^{(-r)}) \]
where $A^{(-r)}$ is identified with $A$ as a ring but the action of $b\in k$ on $A^{(-r)}$ is the same as $b^{p^r}$ acting on $A$ (see \cite[I.9.2]{Jan:2003}). Now consider $G$ as a group scheme. Then, for each $r$ let $F_r:G\to G^{(r)}$ be the $r$-th Frobenius morphism. The scheme-theoretic kernel $G_r=\ker(F_r)$ is called the $r$-th Frobenius kernel of $G$. Given a closed subgroup (scheme) H of G, write $H_r$ for the scheme-theoretic kernel of the restriction $F_r:H\to H^{(r)}$. In other words, we have
\[ H_r=H\cap G_r. \] 
Given a rational $G$-module $M$, write $M^{(r)}$ for the module obtained by twisting the structure map for $M$ by $F_r$. Note that $G_r$ acts trivially on $M^{(r)}$. Conversely, if $N$ is a $G$-module on which $G_r$ acts trivially, then there is a unique $G$-module $M$ with $N = M^{(r)}$. We denote the module $M$ by $N^{(-r)}$.

Now suppose $R\subseteq S$ are finitely generated commutative $k$-algebras. Let $F_r(S)=\{s^{p^r}~|~s\in S\}$ for every $r\ge 1$. If we have the following inclusions 
\[ F_r(S)\subseteq R\subseteq S \]
then the inclusion $R\hookrightarrow S$ induces the homeomorphism from $\spec S$ onto $\spec R$, and we call this map an $F$-isomorphism \cite[5.4]{Ben:1998}.  

\subsection{Geometry}
Let $R$ be a commutative Noetherian ring with identity. We use $R_{\red}$ to denote the reduced ring $R/\rad{R}$ where $\rad{R}$ is the radical ideal of $0$ in $R$, which consists of all nilpotent elements of $R$. Let $\spec R$ be the spectrum of all prime ideals of $R$. This set is a topological space under the Zariski topology. Let $X$ be a variety. We denote by $k[X]$ the algebra of regular functions defined on $X$. Note that when $X$ is an affine variety, $k[X]$ coincides with the coordinate algebra of $X$.

Denote by $\N$ the nilpotent cone of $\g$. There is an adjoint action of $G$ on $\g$ which stablizes $\N$. We call it the dot action and use ``$\cdot$" for the notation. For $H$ a subgroup of $G$, suppose $X, Y$ are $H$-varieties. Then the morphism $f:X\to Y$ is called $H$-equivariant if it is compatible with $H$-action. 

Given a $G$-variety $V$, it can be seen that $B$ acts freely on $G\times V$ by setting $b\cdot(g,v)=(gb^{-1},bv)$ for all $b\in B,g\in G$ and $v\in V$. The notation $G\times^BV$ stands for the fiber bundle associated to the projection $\pi:G\times^BV\rightarrow G/B$ with fiber $V$. Topologically, $G\times^BV$ is a quotient space of $G\times V$ in which the equivalence relation is given as 
\[ (g,v)\sim (g',v')\Leftrightarrow (g',v')=b\cdot(g,v)~~~\text{for some}~~b\in B.\]
In other words, each equivalence class of $G\times^BV$ represents a $B$-orbit in $G\times V$. The map $m:G\times^BV\rightarrow G\cdot V$ defined by mapping $[g,v]$ to $g\cdot v$ for all $g\in G,v\in V$ is called the moment morphism. It is obviously surjective. Although $G\times^BV$ is not affine, we still denote by $k[G\times^BV]$ the ring of global sections on this variety. It is sometimes useful to make the following identification: $k[G\times^BV]\cong k[G\times V]^B$. Note also that $\spec k[G\times^BV]=G\times_BV$.

\section{Main tools}\label{generalsection}

In this section, we introduce some methods to compute $B_r$ and $G_r$-cohomology which will be applied later to obtain results on $(SL_2)_r$-cohomology. For simplicity, we write $S^i$ and $\Lambda^j$ instead of $S^i(\fraku^*)$ and $\Lambda^j(\fraku^*)$. We start with the spectral sequences in \cite[Proposition I.9.14]{Jan:2003}. Replacing $\g^*$ by $\fraku^*$, we immediately get spectral sequences to compute cohomology of $U_r$ with coefficients in a $B$-module $M$. The resulting spectral sequences can be written as follows: If $p\ne 2$, then
\begin{equation}\label{U_rspectralsequence}
E^{i,j}_1=\bigoplus M\otimes S^{a_1(1)}\otimes\cdots\otimes S^{a_r(r)}\otimes\Lambda^{b_1}\otimes\Lambda^{b_2(1)}\otimes\cdots\otimes\Lambda^{b_r(r-1)}\Rightarrow\opH^{i+j}(U_r,M)
\end{equation}
where the direct sum is taken over all $a_i$'s and $b_j$'s satisfying 
\begin{equation}\label{ij-condition}
\left\{ \begin{array}{rcl} i+j  & = & 2(a_1+\cdots+a_r)+b_1+\cdots+b_r \\
 i & = & \sum_{n=1}^{r}(a_np^n+b_np^{n-1}). \end{array}\right.
\end{equation}
If $p=2$, then
\begin{equation}\label{U_rspectralsequence:p=2}
E^{i,j}_1=\bigoplus M\otimes S^{a_1}\otimes\cdots\otimes S^{a_r(r-1)}\Rightarrow\opH^{i+j}(U_r,M)
\end{equation}
where the direct sum is taken over all $a_i$'s satisfying 
\begin{equation}\label{ij-condition:p=2}
\left\{ \begin{array}{rcl} i+j  & = & a_1+\cdots+a_r \\
 i & = & \sum_{n=1}^{r}a_np^{n-1}. \end{array}\right.
\end{equation}
These spectral sequences will play important roles in our calculations for cohomology in later sections.

\subsection{Spectral sequence for $B_r$-cohomology}
We first prove an easy lemma related to Frobenius twist and untwist of a module.
\begin{lemma}\label{easylemma}
Suppose $M$ is a $G$-module. Then for $j\ge i\ge 0$ we have 
\[ M^{T_j/T_i}\cong\left[ \left(M^{(-i)}\right)^{T_{j-i}}\right]^{(i)}.\]
\end{lemma}
\begin{proof}
From \cite[Proposition I.9.5]{Jan:2003}, we can identify $T_j/T_j$ with $T_{j-i}$ via $F_{i}$. Then $T_{j-i}$ acts on $M$ by untwisting the morphism $F_{i}$, that is $M^{(-i)}$. So instead of computing $T_j/T_i$-invariant of $M$, we calculate the $T_{j-i}$-invariant of $M^{(-i)}$ and then twist it back to the $T_j$-module structure; hence the isomorphism follows. 
\end{proof}

In order to set up an inductive proof, we now look at a simple case when $r=2$ and $p\ne 2$. As $B_2/U_2\cong T_2$ is diagonalisable, Corollary I.6.9 in \cite{Jan:2003} gives us
 \[ \opH^{i+j}(U_2,M)^{T_2}\cong\opH^{i+j}(B_2,M).\]
So applying the fixed point functor on both sides of the spectral sequence \eqref{U_rspectralsequence}, we have
\[
\bigoplus\left(M\otimes S^{a_1(1)}\otimes S^{a_2(2)}\otimes\Lambda^{b_1}\otimes\Lambda^{b_2(1)}\right)^{T_2} \Rightarrow\opH^{i+j}(B_2,M) \]
where the direct sum is taken over all $a_i,b_i$'s satisfying the condition \eqref{ij-condition}. As $S^{a_2(2)}$ is $T_2$-invariant and $S^{a_1(1)}\otimes\Lambda^{b_2(1)}$ is $T_1$-invariant, each direct summand on the left hand side is isomorphic to the following
\begin{align*}
S^{a_2(2)}\otimes\left(M\otimes S^{a_1(1)}\otimes \Lambda^{b_1}\otimes\Lambda^{b_2(1)}\right)^{T_2} &\cong S^{a_2(2)}\otimes\left[ \left(M\otimes S^{a_1(1)}\otimes \Lambda^{b_1}\otimes\Lambda^{b_2(1)}\right)^{T_1}\right]^{T_2/T_1}\\
&\cong S^{a_2(2)}\otimes\left(S^{a_1(1)}\otimes\Lambda^{b_2(1)} \otimes(M\otimes\Lambda^{b_1})^{T_1}\right)^{T_2/T_1} \\
&\cong S^{a_2(2)}\otimes\left[\left( S^{a_1}\otimes\Lambda^{b_2}\otimes\left[(M\otimes\Lambda^{b_1})^{T_1}\right]^{(-1)}\right)^{T_1}\right]^{(1)} 
\end{align*}
where the last isomorphism is obtained from Lemma \ref{easylemma}. This computation can be generalized for arbitrary $r$ as follows.
\begin{theorem}\label{B_rspectralsequence} 
There exists, for each $B$-module $M$, a spectral sequence converging to $\opH^\bullet(B_r,M)$ as a $B^{(r)}$-module with the following $E_1$-terms: If $p\ne 2$, then
\begin{align}\label{formE_1}
E^{i,j}_1=\bigoplus S^{a_r(r)}\otimes\left[\left(\left[\left( \left[(M\otimes\Lambda^{b_1})^{T_1}\right]^{(-1)}\otimes S^{a_1}\otimes\Lambda^{b_2}\right)^{T_1}\right]^{(-1)}\otimes..\otimes\Lambda^{b_r}\right)^{T_1}\right]^{(r-1)}
\end{align}
where the direct sum is taken over all $a_i,b_j$ satisfying condition \eqref{ij-condition}. If $p=2$, then
\begin{align}\label{formE_1:p=2}
E^{i,j}_1=\bigoplus \left[\left(\left[\left( \left[(M\otimes S^{a_1})^{T_1}\right]^{(-1)}\otimes S^{a_2}\right)^{T_1}\right]^{(-1)}\otimes..\otimes S^{a_r}\right)^{T_1}\right]^{(r-1)}
\end{align}
where the direct sum is taken over all $a_i$ satisfying condition \eqref{ij-condition:p=2}. \end{theorem}

\begin{proof}
We only give a proof for the case when $p\ne 2$ since that for the case $p=2$ is very similar. First consider the spectral sequence for $U_1$-cohomology as follows:
\[ \bigoplus M\otimes S^{a_1(1)}\otimes\Lambda^{b_1}\Rightarrow \opH^n(U_1,M) \]
Taking $T_1$-invariant functor on both sides, we have
\begin{align*}
 \bigoplus (M\otimes S^{a_1(1)}\otimes\Lambda^{b_1})^{T_1} &\Rightarrow (\opH^n(U_1,M))^{T_1}\cong\opH^n(B_1,M)\\
 \bigoplus S^{a_1(1)}\otimes(M\otimes \Lambda^{b_1})^{T_1} &\Rightarrow \opH^n(B_1,M).
\end{align*}
This verifies the theorem for $r=1$. Suppose it is true for $r$. Apply the invariant functor $(-)^{T_{r+1}}$ on the $U_{r+1}$-spectral sequence, we obtain   
\[ \bigoplus \left( M\otimes S^{a_1(1)}\otimes..\otimes S^{a_{r+1}(r+1)}\otimes\Lambda^{b_1}\otimes..\otimes\Lambda^{b_{r+1}(r)}\right)^{T_{r+1}}\Rightarrow\opH^{i+j}(B_{r+1},M)\]
with $a_i,b_j$ satisfying 
\begin{equation*}
\left\{ \begin{array}{rcl} i+j  & = & 2(a_1+\cdots+a_{r+1})+b_1+\cdots+b_{r+1} \\
 i & = & \sum_{n=1}^{r+1}(a_np^n+b_np^{n-1}). \end{array}\right.
\end{equation*}
In order to complete our induction proof on $r$, we show that the $E_1$-page of this above spectral sequence can be rewritten in the form of \eqref{formE_1}. Indeed, using similar argument as for the case $r=2$ earlier, we have the following isomorphisms for each direct summand of LHS. 
\begin{align*}
& S^{a_{r+1}(r+1)}\otimes\left[ \left( M\otimes S^{a_1(1)}\otimes\cdots\otimes S^{a_r(r)}\otimes\Lambda^{b_1}\otimes\cdots\otimes \Lambda^{b_r(r-1)}\right)^{T_r}\otimes \Lambda^{b_{r+1}(r)}\right]^{T_{r+1}/T_r} \\
&\cong S^{a_{r+1}(r+1)}\otimes\left[ \left( M\otimes S^{a_1(1)}\otimes\cdots\otimes S^{a_{r-1}(r-1)}\otimes\Lambda^{b_1}\otimes\cdots\otimes\Lambda^{b_r(r-1)}\right)^{T_r}\otimes S^{a_r(r)}\otimes\Lambda^{b_{r+1}(r)}\right]^{T_{r+1}/T_r}.
\end{align*}
By applying the inductive hypothesis on the $T_r$-invariant module in the bracket and Lemma \ref{easylemma}, the last module is then isomorphic to the following.
\begin{align*}  
& S^{a_{r+1}(r+1)}\otimes\left[ \left[\left(\left[\left( \left[(M\otimes\Lambda^{b_1})^{T_1}\right]^{(-1)}\otimes S^{a_1}\otimes\Lambda^{b_2}\right)^{T_1}\right]^{(-1)}\otimes..\otimes\Lambda^{b_r}\right)^{T_1}\right]^{(r-1)}\otimes S^{a_r(r)}\otimes\Lambda^{b_{r+1}(r)}\right]^{T_{r+1}/T_r}  \\
&\cong S^{a_{r+1}(r+1)}\otimes\left[ \left(\left[\left(\left[\left( \left[(M\otimes\Lambda^{b_1})^{T_1}\right]^{(-1)}\otimes S^{a_1}\otimes\Lambda^{b_2}\right)^{T_1}\right]^{(-1)}\otimes..\otimes\Lambda^{b_r}\right)^{T_1}\right]^{(-1)}\otimes S^{a_r}\otimes\Lambda^{b_{r+1}}\right)^{T_1}\right]^{(r)}.
\end{align*}
This completes our proof.
\end{proof}

\subsection{Spectral sequence for $G_r$-cohomology}\label{G_rspectralsequence} 
Recall form \cite[II.12.2]{Jan:2003} that if $R^i\ind_B^GM=0$ for all $i>0$, then $G_r$-cohomology can be computed from the following spectral sequence 
\[R^n\ind_B^G(\opH^m(B_r,M)^{(-r)})\Rightarrow\opH^{n+m}(G_r,\ind_B^GM)^{(-r)}.\] 
In particular, for any dominant weight $\lambda\in X^+$, we always have 
\[R^n\ind_B^G(\opH^m(B_r,\lambda)^{(-r)})\Rightarrow\opH^{n+m}(G_r,\opH^0(\lambda))^{(-r)}.\] 
So our strategy is to compute $\opH^m(B_r,\lambda)$ first by Theorem \ref{B_rspectralsequence}, then use this spectral sequence to get $G_r$-cohomology of $\opH^0(\lambda)$.

\section{Cohomology}\label{cohomology}
We assume from now on that $G=SL_2$. Let $\alpha$ be the only simple root in the root system $\Phi$ of $G$. Denote $\omega$ the fundamental weight in the weight lattice $X$ \cite[13.2]{Hum:1978}. Then we have $\omega=\frac{\alpha}{2}$. Note also that $\fraku$ is a one-dimensional vector space so we have the following $T$-module identifications on each degree of the exterior algebra 
\[ \Lambda^i=\Lambda^i(\fraku^*)=\begin{cases}
k & \text{if $i=0$}, \\
\alpha & \text{if $i=1$}, \\
0  &  \text{otherwise}.
\end{cases} \]

This section is organized by the value of $p$. More explicitly, the first three subsections is to deal with the cohomology of $U_r$, $B_r$, and $G_r$ in the case $p\ne 2$ while the last one is to provide results in the case $p=2$. Although same methods are applied for both cases, their results are quite different. The author finds this organization is more convenient than listing 2 cases for each result of the section.

Our overall goal is computing the cohomology $\opH^n(G_r,\opH^0(\lambda))$ for every dominant weight $\lambda\in X^+$. Following the strategy in Subsection \ref{G_rspectralsequence}, we start with $U_r$-cohomology. We assume $p\ne 2$ for the first three subsections.

\subsection{Cohomology of $U_r$}

\begin{proposition}\label{U_rwithlambda}
Let $\lambda$ be a dominant weight. For each $r\ge 1$, there is a $B$-isomorphism
\[ \opH^n(U_r,\lambda)\cong \bigoplus_{n=2(a_1+\cdots+a_r)+b_1+\cdots+b_r} \lambda\otimes S^{a_1(1)}\otimes\cdots\otimes S^{a_r(r)}\otimes\Lambda^{b_1}\otimes\Lambda^{b_2(1)}\otimes\cdots\otimes\Lambda^{b_r(r-1)}. \]
\end{proposition}

\begin{proof}
Recall from Section \ref{generalsection} that $U_r$-cohomology can be computed by the following spectral sequence 
\begin{equation*}
E^{i,j}_1=\bigoplus \lambda\otimes S^{a_1(1)}\otimes\cdots\otimes S^{a_r(r)}\otimes\Lambda^{b_1}\otimes\Lambda^{b_2(1)}\otimes\cdots\otimes\Lambda^{b_r(r-1)}\Rightarrow\opH^{i+j}(U_r,\lambda)
\end{equation*}
where the direct sum is taken over all tuples $(a_1,\ldots,a_r,b_1,\ldots,b_r)\in\mathbb{N}^r\times\{0,1\}^r$ satisfying $i+j=2(a_1+\ldots+a_r)+b_1+\ldots+b_r$ and $i=\sum_{n=1}^{r}(a_np^n+b_np^{n-1})$.
Observe that as a $B$-module, $S^m=m\alpha$ and $\Lambda^0=k,\Lambda^1=\alpha$, and $\Lambda^m=0$ for $m>1$. Consider for each $n>0$, we have $d^{i,j}_n:E^{i,j}_n\rightarrow E^{i+n,j-n+1}_n$ where the $B$-module $E_n^{i,j}$ has weight 
\begin{align*}
\lambda + pa_1\alpha+\cdots+p^ra_r\alpha+b_1\alpha+\cdots+p^{r-1}b_r\alpha &=\lambda +(pa_1+\cdots+p^ra_r+b_1+\cdots+p^{r-1}b_r)\alpha \\
&=\lambda + i\alpha.
\end{align*}
Likewise, $E_n^{i+n,j-n+1}$ is of weight $\lambda + (i+n)\alpha$. As all the differentials respect to $T$-action, we must have $\lambda +i\alpha=\lambda+(i+n)\alpha$ if the map is nonzero. This implies $n=0$, and so $d^{i,j}_n=0$ for all $i,j$ and $n>0$. Hence the spectral sequence collapses at the first page. The result therefore follows.
\end{proof}

When $\lambda=0$, the isomorphism is also compatible with the ring structure. This computation was completely done by Andersen-Jantzen in \cite[2.4]{AJ:1984}. We recall their result as follows.

\begin{corollary}\label{U_rcohoring}
For each $r\ge 1$, there is an isomorphism of $B$-algebras
\[\opH^{\bullet}(U_r,k)\cong S^{\bullet(1)}\otimes\cdots\otimes S^{\bullet(r)}\otimes\Lambda^\bullet\otimes\cdots\otimes\Lambda^{\bullet(r-1)}.\]
Consequently,
\[\opH^{\bullet}(U_r,k)_{\red}\cong S^{\bullet(1)}\otimes\cdots\otimes S^{\bullet(r)}.\]
\end{corollary}

\begin{remark}
Observe that $U$ is isomorphic to $\mathbb{G}_a$. Our method can be applied for $V=\mathbb{G}_a^n$ to obtain the same result as in \cite[Proposition I.4.27(b)]{Jan:2003}. Our approach not only gives a shorter proof but also provides more information on the module structure of $\opH^\bullet(U_r,k)$. It guarantees the isomorphism is compatible with the $B$-module action which will become handy in computing $G_r$-cohomology later.
\end{remark} 

\subsection{Cohomology of $B_r$}\label{B_rcohomology}
For later convenience, we identify $S^{\bullet(i)}$ with the polynomial ring $k[x]^{(i)}=k[x_{i}]$ where $x_i$ has weight $p^i\alpha$ and degree 2. We also denote for each $1\le i\le r$ by $y_{i-1}$ the generator of the exterior algebra $\Lambda^{(i)}$. In particular, we have
\[ S^{\bullet(1)}\otimes\cdots\otimes S^{\bullet(r)}\otimes\Lambda^\bullet\otimes\cdots\otimes\Lambda^{\bullet(r-1)}=k[x_1,\ldots,x_r]\otimes\Lambda(y_0,\ldots,y_{r-1}). \]
Now we make use of this notation to describe the $B_r$-cohomology.
 
\begin{proposition}\label{B_rwithlambda}
Suppose $\lambda$ is a dominant weight in $X^+$, i.e., $\lambda=m\omega$ for some non-negative integer $m$. For each $r\ge 1$, there is a $B$-isomorphism 
\[\opH^n(B_r,\lambda)^{(-r)}\cong\bigoplus\left<x_1^{a_1}y_0^{b_1}x_2^{a_2}y_1^{b_2}\cdots x_r^{a_r}y_{r-1}^{b_r}\right> \]
where the direct sum is taken over all $a_i,b_j$ satisfying the following conditions
\begin{equation}
\begin{cases}\label{ab-condition} 
a_i\in\mathbb{N}~ \mbox{and}~ b_i\in\{0,1\}~ \mbox{for all} ~1\le i\le r \\
n=2(a_1+\cdots+a_r)+b_1+\cdots+b_r \\
\frac{m}{2}+b_1+(a_1+b_2)p+\cdots+a_rp^r\in p^rX^+.
\end{cases}
\end{equation}

\end{proposition}

\begin{proof}
It is observed that collapsing of the spectral sequence \eqref{U_rspectralsequence} implies the collapse of the one in Theorem \ref{B_rspectralsequence}. So Proposition \ref{U_rwithlambda} implies that
\[\opH^n(B_r,\lambda)\cong\bigoplus S^{a_r(r)}\otimes\left[\left(\left[\left( \left[(M\otimes\Lambda^{b_1})^{T_1}\right]^{(-1)}\otimes S^{a_1}\otimes\Lambda^{b_2}\right)^{T_1}\right]^{(-1)}\otimes..\otimes\Lambda^{b_r}\right)^{T_1}\right]^{(r-1)}\]
where the direct sum is taken over all tuples $(a_1,\ldots,a_r,b_1,\ldots,b_r)\in\mathbb{N}^r\times\{0,1\}^r$ satisfying $n=2(a_1+\ldots+a_r)+b_1+\ldots+b_r$. Using the identifications earlier, we can explicitly write out the cohomology for $B_r$ as a decomposition of weight spaces 
\[\opH^n(B_r,\lambda)^{(-r)}\cong\bigoplus\left<x_1^{a_1}y_0^{b_1}x_2^{a_2}y_1^{b_2}\cdots x_r^{a_r}y_{r-1}^{b_r}\right> \]
where each monomial weights 
\[\left[a_r\alpha+\frac{\frac{\frac{\frac{\lambda+b_1\alpha}{p}+(a_1+b_2)\alpha}{p}+(a_2+b_3)\alpha}{p}+\cdots+(a_{r-1}+b_r)\alpha}{p}\right]\in X^+.\]
This is equivalent to $\frac{m}{2}+b_1+(a_1+b_2)p+\cdots+a_rp^r\in p^rX^+$. Now as $U$ is an abelian group in this case, it trivially acts on both sides of the above isomorphism. Hence, this isomorphism is compatible with $B$-action via the identification $B/U\cong T$.
\end{proof}

\begin{remark}\label{remarkB_r}
For each tuple $(a_1,\ldots,a_r,b_1,\ldots,b_r)\in\mathbb{N}^r\times\{0,1\}^r$ satisfying \eqref{ab-condition}, there is $\gamma=n\omega$ for some non-negative integer $n$ such that
\begin{equation}\label{lambda-eqn}
\frac{m}{2}+b_1+(a_1+b_2)p+\cdots+(a_{r-1}+b_r)p^{r-1}+a_rp^r=\frac{n}{2}p^r.
\end{equation}
Let $N_r(p,m,n)$ denote the number of solutions $(a_1,\ldots,a_r,b_1,\ldots,b_r)\in\mathbb{N}^r\times\{0,1\}^r$ for the equation. As a consequence, we establish our goal of this subsection.
\end{remark}

\begin{theorem}\label{character2}
Suppose $\lambda=m\omega\in X^+$. Then there is a $B$-module isomorphism
\[ \opH^\bullet(B_r,\lambda)^{(-r)}\cong\bigoplus_{n=0}^\infty (n\omega)^{N_r(p,m,n)} \] and so
\[ \ch\opH^\bullet(B_r,\lambda)^{(-r)}=\sum_{n\in\mathbb{N}}N_r(p,m,n)e(n\omega). \]
\end{theorem}

\subsection{Cohomology of $G_r$}
We are now ready to compute $G_r$-cohomology.

\begin{theorem}\label{G_rwithlambda}
Suppose $\lambda=m\omega\in X^+$. Then there are $G$-module isomorphisms
\begin{align*}
\opH^\bullet(G_r,\opH^0(\lambda))^{(-r)}\cong\ind_B^G(\opH^\bullet(B_r,\lambda)^{(-r)})\cong
\bigoplus_{n=0}^\infty\ind_B^G\left(n\omega\right)^{N_r(p,m,n)}.
\end{align*}
where each $N_r(p,m,n)$ is defined in Remark \ref{remarkB_r}. In particular, if $\lambda=0$, then the first isomorphism is compatible with the cup-products on both sides, i.e., it is an isomorphism of graded $G$-algebras. 
\end{theorem}

\begin{proof}
Recall the spectral sequence in Section \ref{G_rspectralsequence}, we have
\[E^{m,n}_2=R^m\ind_B^G(\opH^n(B_r,\lambda)^{(-r)})\Rightarrow\opH^{n+m}(G_r,\opH^0(\lambda))^{(-r)}.\] 
Note that all weights of $\opH^\bullet(B_r,\lambda)$ are also weights of $\opH^\bullet(U_r,\lambda)$ which are in turn weights of $\lambda\otimes S^{\bullet(1)}\otimes\cdots\otimes S^{\bullet(r)}\otimes\Lambda^\bullet\otimes\Lambda^{\bullet(1)}\otimes\cdots\otimes\Lambda^{\bullet(r-1)}$ by the spectral sequence \eqref{U_rspectralsequence}. So all weights of $\opH^n(B_r,\lambda)$ are dominant for each $n\ge 0$. It follows from Kempf's vanishing that $R^m\ind_B^G(\opH^n(B_r,\lambda))=0$ for each $n\ge 0$ and $m>0$. Hence the spectral sequence collapses at the first page and we obtain 
\begin{equation*}
\opH^{n}(G_r,\opH^0(\lambda))^{(-r)}\cong\ind_B^G\left(\opH^n(B_r,\lambda)^{(-r)}\right).
\end{equation*}
This implies the first isomorphism in the theorem, that is, 
\[ \opH^\bullet(G_r,\opH^0(\lambda))^{(-r)}\cong\ind_B^G(\opH^\bullet(B_r,\lambda)^{(-r)}). \]
Moreover, by \cite[Remark 3.2]{AJ:1984}, this isomorphism respects the graded $\opH^{n}(G_r,k)$-module structure on both sides. Hence, in the case $\lambda=0$ it is a graded $G$-algebra isomorphism. The other isomorphism follows by Theorem \ref{character2}.
\end{proof}

\begin{remark}
This theorem gives us an explicit good filtration of $\opH^\bullet(G_r,\opH^0(\lambda))^{(-r)}$; thus, showing the same property for each $\opH^n(G_r,\opH^0(\lambda))^{(-r)}$ with $n\ge 0, r\ge 1$. This result is similar to the one in \cite[4.5(1)]{AJ:1984} and \cite[Corollary 2.2]{vdK:1993}. In fact, for arbitrary simple algebraic group $G$, it is conjectured that $\opH^n(G_r,\opH^0(\lambda))$ has a good filtration \cite[12.15]{Jan:2003}.
\end{remark}

Following Theorem \ref{character2}, we immediately obtain the analogous result for $G_r$-cohomology.
\begin{corollary}
Suppose $\lambda=m\omega\in X^+$. Then we have
\begin{align*}
\ch\opH^\bullet(G_r,\opH^0(\lambda))^{(-r)} &=\sum_{n\in\mathbb{N}}N_r(p,m,n)\ch\opH^0(n\omega).
\end{align*}
\end{corollary} 

\subsection{The case $p=2$} The results in this subsection are analogs of those in previous subsections. Most of proofs will be omitted.

\begin{proposition}\label{U_rwithlambda:p=2}
Let $\lambda$ be a dominant weight. For each $r\ge 1$, there is a $B$-isomorphism
\[ \opH^n(U_r,\lambda)\cong \bigoplus_{n=a_1+\cdots+a_r} \lambda\otimes S^{a_1}\otimes S^{a_2(1)}\cdots\otimes S^{a_r(r-1)}. \]
Consequently, there is an isomorphism of $B$-algebras
\[ \opH^\bullet(U_r,k)\cong S^\bullet\otimes S^{\bullet(1)}\otimes\cdots\otimes S^{\bullet(r-1)}. \]
\end{proposition}
\begin{proof}
Same argument as in the proof of Proposition \ref{U_rwithlambda} is applied for the spectral sequence \ref{U_rspectralsequence:p=2}.
\end{proof}
\begin{proposition}\label{B_rwithlambda:p=2}
Suppose $\lambda$ is a dominant weight in $X^+$, i.e., $\lambda=m\omega$ for some non-negative integer $m$. For each $r\ge 1$, there is a $B$-isomorphism 
\[\opH^n(B_r,\lambda)^{(-r)}\cong\bigoplus\left<x_1^{a_1}x_2^{a_2}\cdots x_r^{a_r}\right> \]
where the direct sum is taken over all $a_i$ satisfying the following conditions
\begin{equation}
\begin{cases}\label{ab-condition} 
a_i\in\mathbb{N}~~ \mbox{for all} ~1\le i\le r \\
n=a_1+\cdots+a_r \\
m+2a_1+4a_2+\cdots+2^ra_r\in 2^rX^+.
\end{cases}
\end{equation}
\end{proposition}
\begin{proof}
Same argument as in the proof of Proposition \ref{B_rwithlambda} is applied for the spectral sequence \ref{formE_1:p=2}.
\end{proof}

\begin{theorem}\label{G_rwithlambda:p=2}
Suppose $\lambda=m\omega\in X^+$. Then there are $G$-module isomorphisms
\begin{align*}
\opH^\bullet(G_r,\opH^0(\lambda))^{(-r)}\cong\ind_B^G(\opH^\bullet(B_r,\lambda)^{(-r)})\cong
\bigoplus_{n=0}^\infty\ind_B^G\left(n\omega\right)^{N_r(2,m,n)}.
\end{align*}
where $N_r(2,m,n)$ is the number of $r$-tuples $(a_1,\ldots,a_r)\in\mathbb{N}^r$ satisfying the equation
\[ m+2a_1+4a_2+\cdots+2^ra_r=2^rn.\] 
In particular, if $\lambda=0$, then the first isomorphism is compatible with the cup-products on both sides, i.e., it is an isomorphism of graded $G$-algebras. 
\end{theorem}

\section{An algorithm to compute $N_r(p,m,n)$}\label{character multiplicity} 
\subsection{} Results in the preceding section indicate that the number $N_r(p,m,n)$ plays an important role in the cohomology of $B_r$ and $G_r$. It is closely related to the problem of counting integral lattice points in a polytope. In particular, let $P_r$ be a polytope in $\mathbb{R}^r$ determined by the following equations:
\[
\begin{cases}
\frac{m}{2}+y_1+(x_1+y_2)p+\cdots+(x_{r-1}+y_r)p^{r-1}+x_rp^r=\frac{n}{2}p^r,\\
0\le x_i\le\frac{n}{2}p^r ~\mbox{for each}~1\le i\le r, \\
0\le y_i\le 1~\mbox{for each}~1\le i\le r.
\end{cases}
\]
Then for $p\ne 2$ and $m,n\ge 0$ we have $N_r(p,m,n)=|P_r\cap\Z^r|$ for each $r\ge 1$. So one can use Barvinok's algorithm to compute the right-hand side. However, this algorithm is getting slow when $r$ is big due to many complicated subalgorithms and computations involving Complex Analysis. Barvinok actually proved that the algorithm terminates after a polynomial time for a fixed dimension depending on the data for vertices of the polytope (see \cite{Bar:1994} for further details). 

\subsection{} In this subsection, we sketch an alternative program to calculate the number $N_r(p,m,n)$ with given non-negative integers $m,n$; $r\ge 2$ and a prime $p>2$.
We start with an algorithm to compute the number of solutions $(c_1,...,c_r)$ satisfying the equation
\begin{equation}\label{c-eqn}
np^r=(c_1+d_1)p+...+(c_r+d_r)p^r
\end{equation}
with given nonnegative integer $d_i$ for each $i$.
Define recursively a family of functions $N_i:p^i\mathbb{N}\rightarrow\mathbb{N}$ as follow:
\begin{itemize}
\item $N_i(0)=1$ for each $1\le i\le r$
\item For each $n\in\mathbb{N}$, $N_1(np)=\left\{ \begin{array}{rcl} 0 & \mbox{if} & d_1>n \\
1 & \mbox{otherwise} \end{array}\right. $
\item For each $n\in\mathbb{N}$, $N_2(np^2)=\left\{ \begin{array}{rcl} 0 & \mbox{if} & d_2>n \\
\Sigma_{j=0}^{n-d_2}N_1(jp^2) & \mbox{otherwise} \end{array}\right. $
\item For each $n\in\mathbb{N}$, $N_{i+1}(np^{i+1})=\left\{ \begin{array}{rcl} 0 & \mbox{if} & d_{i+1}>n \\
\Sigma_{j=0}^{n-d_{i+1}}N_i(jp^{i+1}) & \mbox{otherwise} \end{array}\right. $
\end{itemize}
\begin{theorem}
For each $n\in\mathbb{N}$, the number of solutions of equation \eqref{c-eqn} is $N_r(np^r)$.
\end{theorem}
\begin{proof}
We prove by induction on $r$. It is trivial for $r=1$. Suppose it is true for $r-1$. From equation \eqref{c-eqn}, we have $n-d_r+1$ choices for $c_r\in\{d_r,d_r+1,..,n\}$. For each $c_r=i$, the number of solutions is $N_{r-1}(np^r-ip^r)=N_{r-1}[p(n-i)p^{r-1}]$. Summing all the terms and using the recursive formula, we get the total number of solutions $N_r(np^r)$.
\end{proof}
Now the algorithm consists of following steps:
\begin{itemize}
\item If $m$ is even and $n$ is odd, or if $m$ is odd and $n$ is even then the function returns $0$.
\item If both $m,n$ are even, then let $m'=\frac{m}{2}$ and $n'=\frac{n}{2}$. The equation \ref{lambda-eqn} hence becomes
\begin{equation}
m'+b_1+(a_1+b_2)p+...+(a_{r-1}+b_r)p^{r-1}+a_rp^r=n'p^r.
\end{equation}
\begin{enumalph}
\item Write $m'$ into base $p$-expansion. Suppose $m'=d_0+d_1p+...+d_hp^h$ for some $h\in\mathbb{N}$. Note that if $b_1+m'>n'p^r$ then, of course, $N_r(p,m,n)=0$.
\item For each $r$-tuple $(b_1,...,b_r)\in\{0,1\}^r$, we use previous theorem to compute number $N_r(m',n',b_1,...,b_r)$ of solutions for
\[ d_0+b_1+(a_1+d_1+b_2)p+...+(a_{r-1}+d_{r-1}+b_r)p^{r-1}+(a_r+d_r)p^r=(n'-d_{r+1}p-...-d_hp^{h-r})p^r. \]
\item We have $N_r(p,m,n)=\Sigma_{b_1,..,b_r}N_r(m',n',b_1,...,b_r)$.
\end{enumalph}
\item If both $m$ and $n$ are odd, then repeat Step 2 with $m'=\frac{m+1}{2}$ and $n'=\frac{n+p^r}{2}$.
\end{itemize}
\begin{remark}
This algorithm was coded and compared with the one coded by applying Erhart's theory. The tables in Appendix \ref{appendix} shows that our program runs faster than the other.
\end{remark}

\section{Reduced rings and geometry}\label{reduced cohomology}
\subsection{Reduced $B_r$-cohomology ring}\label{reduceB_rcoho} 
In Theorem \ref{character2}, we computed $\opH^\bullet(B_r,k)$ as a $B$-module. Describing its ring structure, however, is extremely hard for big $r$ (see \cite{AJ:1984}). By looking at the reduced part, we can compute $\opH^\bullet(B_r,k)$ as a finitely generated $\opH^\bullet(U_r,k)_{\red}$-module. We first need the following observation.
\begin{lemma}\label{lemma of reduceB_r}
Given any $T$-algebra $M$, there is an isomorphism
\[ (M^{T_r})_{\red}\cong (M_{\red})^{T_r}. \]
\end{lemma}
\begin{proof}
Note that $M_{\red}=M/\rad{M}$, we consider the short exact sequence 
\[ 0\rightarrow\rad{M}\rightarrow M\rightarrow M_{\red}\rightarrow 0.\]
As $T_r$ is diagonalisable, $(-)^{T_r}$ is exact, so we obtain that
\[ 0\rightarrow(\rad{M})^{T_r}\rightarrow(M)^{T_r}\rightarrow(M_{\red})^{T_r}\rightarrow 0.\]
On the other hand, we know that $(M^{T_r})_{\red}=M^{T_r}/\rad{(M^{T_r})}$. So we just need to check that $(\rad{M})^{T_r}=\rad{(M^{T_r})}$ which is true since both equal to $\rad{(M)}\cap M^{T_r}$.
\end{proof}
From Corollary \ref{U_rcohoring} (respectively Proposition \ref{U_rwithlambda:p=2} in the case $p=2$), we can identify $\opH^\bullet(U_r,k)_{\red}$ with the polynomial algebra $k[x_1,\ldots,x_r]$ where each $x_i$ is of weight $p^i\alpha$ (respectively $p^{i-1}\alpha$). The preceding lemma and Corollary \ref{U_rcohoring} imply that 
\begin{align*}
\opH^\bullet(B_r,k)_{\red} &\cong\left(\opH^{\bullet}(U_r,k)^{T_r}\right)_{\red}\cong\left(\opH^{\bullet}(U_r,k)_{\red}\right)^{T_r}\\
&\cong\left(S^{\bullet(1)}\otimes\cdots\otimes S^{\bullet(r)}\right)^{T_r}\\
&\cong\left(k[x_1,\ldots,x_r]\right)^{T_r}.
\end{align*}
As a $B$-module, this reduced cohomology ring can be represented in terms of monomials as follows.

\begin{theorem}\label{B_rcohoreduce}
For $r\ge 1$, there is a $B^{(r)}$-module isomorphism
\begin{equation}\label{iso-of-B_rcohoreduce}
\opH^\bullet(B_r,k)_{\red}\cong\bigoplus\left< x_1^{a_1}, x_2^{a_2},\ldots,x_r^{a_r}\right>
\end{equation}
where the $a_i$ are non-negative integers satisfying 
\begin{equation}\label{reduced-eqn}
a_1+a_2p+\cdots+a_rp^{r-1}=np^{r-1}
\end{equation}
for some $n\in\mathbb{N}$. Furthermore, let $R=k[x_1^{p^{r-1}},x_2^{p^{r-2}},\ldots,x_r]$, then the reduced ring $\opH^\bullet(B_r,k)_{\red}$ is a finitely free $R$-module with the basis $\mathfrak{B}_r=\{ x_1^{a_1}, x_2^{a_2},\ldots,x_{r-1}^{a_{r-1}}\}$ where for each $i=1,\ldots,r-1$, $0\le a_i< p^{r-i}$ and satifies the equation \eqref{reduced-eqn}.
\end{theorem}

\begin{proof}
The right-hand side is obtained by setting $m=b_1=b_2=\cdots=b_r=0$ in the Proposition \ref{B_rwithlambda} (respectively setting $m=0$ in Proposition \ref{B_rwithlambda:p=2}). Now observe that every tuple $(m_1p^{r-1},m_2p^{r-2},\ldots,m_r)$ with $m_i$ a non-negative integer is a solution for \eqref{reduced-eqn}. It follows that $\opH^\bullet(B_r,k)_{\red}$ contains $R$ as a subring. Moreover, it can be verified that every monomial in the right-hand side of the isomorphism \eqref{iso-of-B_rcohoreduce} is uniquely written as a product of an element in $R$ and a monomial in $\mathfrak{B}_r$. The fact that $\mathfrak{B}_r$ is finite completes our proof. 

\end{proof}


\begin{corollary}\label{B_rscheme}
For each $r\ge 1$, there is a homeomorphism from $\spec k[x_1,\ldots,x_r]^{(r)}$ onto $\spec\opH^\bullet(B_r,k)_{\red}$.
\end{corollary}

\begin{proof}
We recall the $r$-th Frobenius homomorphism of rings
\begin{align*}
\calF: k[x_1,\ldots,x_r]^{(r)} &\rightarrow k[x_1,\ldots,x_r] \\
\calF(f)&\longmapsto f^{p^{r}}
\end{align*}
for all $f\in k[x_1,\ldots,x_r]$. Observe that 
\[ \Hbul(B_r,k)_{\red}\subseteq k[x_1,\ldots,x_r]^{(r)}\]
are finitely generated commutative algebras. Note further that $\im\calF=k[x_1^{p^{r}},\ldots,x_r^{p^{r}}]$ lies in the ring $R=k[x_1^{p^{r-1}},x_2^{p^{r-2}},\ldots,x_r]$; hence is a subalgebra of $\Hbul(B_r,k)_{\red}$. The inclusions 
\[ \im\calF\subseteq \Hbul(B_r,k)_{\red}\subseteq k[x_1,\ldots, x_r]^{(r)} \]
implies that the morphism $\mathfrak{i}:\spec k[x_1,\ldots, x_r]^{(r)}\to \spec\Hbul(B_r,k)_{\red}$ is an $F$-isomorphism. Hence, it is a homeomorphism.
\end{proof}

\begin{remark}
One can easily construct an example where this morphism is not an isomorphism. For instance, when $p\ne 2$, from the computation of Andersen and Jantzen \cite[2.4]{AJ:1984} we have $\spec\opH^\bullet(B_2,k)_{\red}=\spec k[x_1^p,x_2]$ which is obviously not isomorphic to $\spec k[x_1,x_2]$ as there is no degree one morphism from one to the other.
  
\end{remark}

\subsection{Reduced $G_r$-cohomology ring} We first develop some techniques to compute reduced rings in general context.
\begin{lemma}\label{nilradical}
Let $k$ be a perfect field. Suppose $G$ is a split reductive group over $k$ and let $M$ be a $k$-algebra. Then we have $\rad(k[G]\otimes M)=k[G]\otimes\rad{M}$.  
\end{lemma}

\begin{proof}
Consider the short exact sequence
\[ 0\to\rad{M}\to M\to M_{\red}\to 0. \]
As $G$ is reductive, the coordinate algebra $k[G]$ is free over $k$ \cite[II.1.1]{Jan:2003}. So we have the following sequence
\[ 0\to k[G]\otimes\rad{M}\to k[G]\otimes M\to k[G]\otimes M_{\red}\to 0 \]
is exact. It follows that 
\[ k[G]\otimes M_{\red}\cong\frac{k[G]\otimes M}{k[G]\otimes\rad{M}}.\]
On the other hand, since $k[G]$ is reduced, it is well-known that the ring $k[G]\otimes M_{\red}$ is reduced when $k$ is perfect. This implies that $k[G]\otimes\rad{M}=\rad{(k[G]\otimes M)}$. 
\end{proof}

\begin{lemma}\label{lemma of reduceG_r}
Suppose the same assumptions on $k$ and $G$ as in Lemma \ref{nilradical}. Given a $B$-algebra $M$ and suppose that all weights of $\rad{M}$ are dominant. Then, as an algebra, we always have $\left[\ind_B^GM\right]_{\red}\cong\ind_B^G(M_{\red})$.
\end{lemma}
\begin{proof}
We first show that $\left[(k[G]\otimes M)^B\right]_{\red}\cong\left[(k[G]\otimes M)_{\red}\right]^B$. Let $A=k[G]\otimes M$. Then we need to prove $(A^B)_{\red}\cong(A_{\red})^B$, i.e., $\frac{A^B}{\rad(A^B)}\cong\left(\frac{A}{\rad(A)}\right)^B$. This is equivalent to showing that the following sequence
\[0\rightarrow\rad(A)^B\rightarrow A^B\rightarrow (A_{\red})^B\]
is right exact; hence equivalent to $\opH^1(B,\rad(A))=0$. Indeed, the preceding lemma shows that
\[ \rad(A)=\rad(k[G]\otimes M)=k[G]\otimes\rad(M). \]
It follows by \cite[Proposition I.4.10]{Jan:2003} and Kempf's vanishing that
\[\opH^1(B,\rad(A))=\opH^1\left(B,k[G]\otimes\rad(M)\right)\cong R^1\ind_B^G\left(\rad(M)\right)=0\]
since all weights of $\rad{M}$ are dominant. Finally, we have
\begin{align*} 
\left(\ind_B^GM\right)_{\red} &\cong\left[(k[G]\otimes M)^B\right]_{\red} \\
&\cong\left[(k[G]\otimes M)_{\red}\right]^B \\
&\cong [k[G]\otimes M_{\red}]^B \\
&=\ind_B^G(M_{\red}).
\end{align*}
\end{proof}
Now we are back to the assumption $G=SL_2$. The following result provides a link between the reduced parts of $B_r$- and of $G_r$-cohomology.

\begin{theorem}\label{G_rcohoreduce}
For each $r\ge 1$, there is a $G$-isomorphism \[ \opH^\bullet(G_r,k)_{\red}^{(-r)}\cong\ind_B^G\opH^\bullet(B_r,k)_{\red}^{(-r)}.\]
\end{theorem}

\begin{proof}
It immediately follows from Theorem \ref{G_rwithlambda} and the lemma above.
\end{proof}

\begin{proposition}\label{G_rscheme}
For each $r\ge 1$, there is a homeomorphism from $\spec k[G\times^B\fraku^r]$ onto $\spec\opH^\bullet(G_r,k)_{\red}$.
\end{proposition}

\begin{proof}
Let $\fraku^{(1)}\times\cdots\times\fraku^{(r)}=\spec k[x_1,\ldots, x_r]$ where each $\fraku^{(i)}$ is identified with the weight space $k_{p^i\alpha}$; hence we consider it as an affine space equipped with the $B$-action. In Corollary \ref{B_rscheme}, we can see that the $F$-isomorphism $\mathfrak{i}$ arises from the inclusion $\Hbul(B_r,k)_{\red}\subseteq k[x_1, \ldots, x_r]^{(r)}$, so it is compatible with the $B$-action. It follows the $B$-equivariant homeomorphism
\[ \mathfrak{i}^{(-r)}: \fraku^{(1)}\times\cdots\times\fraku^{(r)}\to \spec\Hbul(B_r,k)_{\red}^{(-r)}. \]

Apply the fibered product with $G$ over $B$ on both sides, we have a homeomorphism
\[ id_G\times_B \mathfrak{i}^{(-r)}: G\times_{B}\left(\fraku^{(1)}\times\cdots\times\fraku^{(r)}\right) \to G\times_{B}\spec\opH^\bullet(B_r,k)_{\red}^{(-r)}. \]
On the other hand, define a map:
\begin{align*}
\Phi:\fraku^r &\to \fraku^{(1)}\times\cdots\times\fraku^{(r)}\\
(y_1, y_2, \ldots , y_r) &\mapsto (y_1^{p}, y_2^{p^{2}},\ldots, y_r^{p^r}) 
\end{align*}
for all $y_i\in\fraku$. It is easy to see that $\Phi$ is a $B$-equivariant continuous map which is also a homeomorphism. It follows that the fibered products $G\times_{B}\left(\fraku^{(1)}\times\cdots\times\fraku^{(r)}\right)$ and $G\times_B\fraku^r$ are homeomorphic as a topological space. Now combine two above homeomorphisms and apply Theorem \ref{G_rcohoreduce}, we establish a homeomorphism from $\spec k[G\times^B\fraku^r]$ onto $\spec\opH^\bullet(G_r,k)_{\red}$; hence completes our proof.
\end{proof}

\section{Cohen-Macaulay Cohomology ring}\label{Cohen-Macaulayrings}
It is known that there are many classes of groups for which the cohomology is Cohen-Macaulay \cite{Ben:2004}. Our calculations in this section provide an evidence for the conjecture that simple algebraic groups also have Cohen-Macaulay cohomology. In particular, we show that the reduced rings $\opH^{\bullet}(B_r,k)_{\red}$ and $\opH^{\bullet}(G_r,k)_{\red}$ are Cohen-Macaulay. These results are true for any characteristic $p>0$, but we only present the proofs in the case $p\ne 2$. As the strategy is the same for the other case, we leave technical detail for readers.

\subsection{} First we restate some facts in \cite{Jan:2004}.
\begin{lemma}
Suppose $\lambda\in\Z\Phi$. Then for each $n\in\Z$ there is an $G$-isomorphism 
\[ \ind_B^G(S^n(\fraku^*)\otimes\lambda)\cong k[G\times^B\fraku]_{(2n+2ht(\lambda))}.\]
Moreover, the algebra $\ind_B^G(S^n(\fraku^*)\otimes\lambda)$ is $k[G\times^B\fraku]$ shifted by the degree $2\rm{ht}(\lambda)$. 
\end{lemma}

\begin{proof}
This is a consequence of Proposition 8.22 in \cite{Jan:2004}. Note that since $\lambda\in\Z\Phi$, we have $\bar{\lambda}=0$. Then replace $k[\hat{\calO},0]=k[\calO]$ by $k[G\times^B\fraku]$ in the formula in page 111.
\end{proof}

\begin{lemma}
For each $\lambda\in\Z\Phi$, the algebra $\ind_B^G\left(S^\bullet(\fraku^{*(1)}\times\cdots\times\fraku^{*(r)})\otimes\lambda\right)$ is isomorphic to $\ind_B^GS^\bullet(\fraku^{*(1)}\times\cdots\times\fraku^{*(r)})$ with the degree shifted by $2\rm{ht}(\lambda)$.
\end{lemma}

\begin{proof}
We employ the algebra map in the previous lemma. In particular, for each degree $n$, we first consider
\begin{align*}
\ind_B^G\left(S^n(\fraku^{*(1)}\times\cdots\times\fraku^{*(r)})\otimes\lambda\right) &= \bigoplus_{a_1+\cdots +a_r=n}\ind_B^G\left(S^{a_1}(\fraku^{*(1)})\otimes\cdots\otimes S^{a_r}(\fraku^{*(r)})\otimes\lambda\right) \\
&\cong \bigoplus_{a_1+\cdots +a_r=n}\ind_B^G\left((a_1p+\cdots +a_rp^r)\alpha+\lambda\right) \\
&=\bigoplus_{a_1+\cdots +a_r=n}\ind_B^G\left(S^{(a_1p+\cdots +a_rp^r)}(\fraku^*)\otimes\lambda\right) \\
&\cong\bigoplus_{a_1+\cdots +a_r=n}k[G\times^B\fraku]_{2(a_1p+\cdots +a_rp^r)+2ht(\lambda)}.
\end{align*}
Now if $\lambda=0$ in this observation, we obtain for each $n$ that
\[ \ind_B^G\left(S^n(\fraku^{*(1)}\times\cdots\times\fraku^{*(r)})\right)\cong\bigoplus_{a_1+\cdots +a_r=n}k[G\times^B\fraku]_{2(a_1p+\cdots +a_rp^r)}. \] 
This implies our proof. Consequently, the two algebras in the lemma are isomorphic (without the grading). 
\end{proof}

\begin{lemma}
The ring $R=\ind_B^G\left(S^\bullet(\fraku^{*(1)}\times\cdots\times\fraku^{*(r)})\right)$ is Cohen-Macaulay.
\end{lemma}

\begin{proof}
It is not hard to see that $k[\fraku^r]$ is a free $k[\fraku^{*(1)}\times\cdots\times\fraku^{*(r)}]$-module. As the induction functor preserves direct sums, we have $k[G\times^B\fraku^r]$ is a free $R$-module. Hence there is a flat homomorphism of $R$-modules
\[ R\hookrightarrow k[G\times^B\fraku^r]\]
By \cite[Theorem 5.2.7 - 5.4.1]{Ngo:2012}, we obtain that the ring $k[G\times^B\fraku^r]$ is Cohen-Macaulay. As the flatness is locally preserved, Proposition 2.6(d) in \cite{I:2008} implies that the Cohen-Macaulayness of $R$ follows from that of $k[G\times^B\fraku^r]$. 
\end{proof}
We can now establish the goal of this section.
\begin{theorem}\label{Cohen-Macaulay for B_r and G_r}
Both rings $\opH^{\bullet}(B_r,k)_{\red}$ and $\opH^{\bullet}(G_r,k)_{\red}$ are Cohen-Macaulay.
\end{theorem}

\begin{proof}
Note that $\opH^{\bullet}(B_r,k)_{\red}=k[\fraku^{(1)}\times\cdots\times\fraku^{(r)}]^{T_r}$, an invariant of a regular domain by a finite group scheme $T_r$. So it is Cohen-Macaulay by a famous result of Hochster and Roberts \cite{HR:1974}.

Next by Theorem \ref{G_rcohoreduce}, we have the decomposition of $R$-modules
\begin{align*}
\opH^{\bullet}(G_r,k)_{\red} &\cong\ind_B^G\opH^{\bullet}(B_r,k)_{\red} \\
&\cong \bigoplus_{x_{\lambda}\in\mathfrak{B}_r}\ind_B^G\left( x_\lambda\otimes k[x_1,\ldots,x_r]\right).
\end{align*}
In other words, the ring $\opH^{\bullet}(G_r,k)_{\red}$ is a free $R$-module. By the previous lemma, $R$ is Cohen-Macaulay, so is $\opH^{\bullet}(G_r,k)_{\red}$.  
\end{proof}

\subsection{Open questions}
The results above imply many open questions involving the properties in commutative algebra like Cohen-Macaulayness for the objects in representation theory as follows.
\begin{conjecture}
Suppose $R$ is a $B$-algebra. If $R$ is Cohen-Macaulay, then so is $\ind_B^GR$.
\end{conjecture}

\begin{conjecture}
Let $G$ be a simple algebraic group defined over an algebraically closed field $k$ of characteristic $p$, a good prime for $G$. Then both $\opH^{\bullet}(B_r,k)_{\red}$ and $\opH^{\bullet}(G_r,k)_{\red}$ are Cohen-Macaulay. 
\end{conjecture}

Evidently, if $r=1$ then the results of Andersen and Jantzen \cite{AJ:1984}, Friedlander and Parshall \cite{FP:1986} show that $\opH^{\bullet}(B_r,k)_{\red}\cong S^\bullet(\fraku^*)$ is regular and $\opH^{\bullet}(G_r,k)_{\red}\cong k[\N]$ is Cohen-Macaulay. Although our computation supports the conjecture in the case $G=SL_2$ for arbitrary $r$, it is a difficult problem as very little appears to be known about cohomology of higher Frobenius kernels.

\section{Quantum calculations}
In this section, we apply our methods to compute cohomology for the Frobenius--Lusztig kernels of quantum groups defined in \cite{Dru:2011}. We first recall the definitions as follows. (Details can be found in \cite[2.1-2.2]{Dru:2011}.)

\subsection{Notation and Construction}\label{quantum-notation} We basically recover the material in \cite[Section 2.1]{Dru:2011} in the case $\g=\fraksl_2$. Let $k$ be a field of characteristic $p\ne 2$. Let $\ell$ be an odd positive integer not divisible by $p$. Denote by $\Phi=\{\alpha,-\alpha\}$ the root system of $\fraksl_2$.  

Let $q$ be an indeterminate and let $\Uq$ be the quantized enveloping algebra associated to $\fraksl_2$, which is the $\mathbb{Q}(q)$-algebra defined by the generators $E, F,K,K^{-1}$ and satisfying the relations 
\begin{align*}
KK^{-1}=1 &=K^{-1}K,\\
KEK^{-1} &= q^2E, \\
KFK^{-1} &= q^{-2}F, \\
EF-FE &=\frac{K - K^{-1}}{q-q^{-1}}.
\end{align*}

Set $\upA=\Z[q,q^{-1}]$. For given integer $i$, set
\[ [i]=\frac{q^i-q^{-i}}{q-q^{-1}},\]
then denote $[i]^!=[i][i-1]\ldots[1]$. Note that $[0]^!=1$ as a convention. Suppose $m$ is a positive integer and $n$ is an integer, we write
\[ \left[ \begin{array}{c} n \\ m \end{array} \right]=\frac{[n][n-1]\dots[n-m+1]}{[1][2]\ldots[m]} \]
where $\left[ \begin{array}{c} n \\ 0 \end{array} \right]=1$. Next, for each $m\ge 0$, we define the $m$-th divided powers as follows:
\[ E^{(m)}=\frac{E^m}{[m]^!}\quad,\quad F^{(m)}=\frac{F^m}{[m]^!}. \]
Now the Lusztig $\upA$-form quantized enveloping algebra $\Ua$ is defined as the $\upA$-subalgebra of $\Uq$ generated by $\{E^{(n)},F^{(n)},K^{\pm 1}:~n\in\mathbb{N}\}$. Fix $\zeta\in k$ a primitive $\ell$-th root of unity in $k$. We consider $k$ as $\upA$-algebra by the homomorphism $\Z[q,q^{-1}]\rightarrow k$ mapping $q\mapsto\zeta$. Let
\[ \Uz=\frac{U_k}{\left< K^{\ell}\otimes 1 - 1\otimes 1\right>} \]
where $U_k=\Ua\otimes_{\upA}k$. Denote by $u_k$ the Hopf subalgebra of $U_k$ generated by $\{ E, F, K \}$. Let $\uz$ be the image of $u_k$ in $\Uz$, and call it the small quantum group.

For each $r\ge 1$, we define $\Uz(G_r)$ to be the subalgebra of $\Uz$ generated by 
\[ \{E, E^{(p^i\ell)}, F, F^{(p^i\ell)}, K:~0\le i\le r-1\},\]
and call it the $r$-th Frobenius-Lusztig kernel of $\Uz$. Note that if $p=0$ then we obtain for every $r\ge 1$ that $\Uz(G_r)=\uz$. We also define for each $r\ge 1$
\begin{align*}
\Uz(B_r) &= \Uz(B)\cap\Uz(G_r) = \left< E , E^{(p^i\ell)}, K~:~0\le i\le r-1\right>, \\
\Uz(U_r) &= \Uz(U)\cap\Uz(G_r) = \left< E, E^{(p^i\ell)}~:~0\le i\le r-1\right>.
\end{align*}

Let $G=SL_2$, a group scheme defined over $k$. Let $\Dist(G)$ be the algebra of distributions on $G$. It is known that there is an isomorphism of Hopf algebras between $\Uz//\uz$ and $\Dist(G)$. The quotient map $F_{\zeta}:\Uz\to\Dist(G)$ is the quantum analog of the Frobenius homomorphism. Note that the restriction $F_{\zeta}:\Uz(T_r)\to\Dist(T_r)$ for each $r\ge 1$ induces an isomorphism $\Uz(T_r)//\uz^0\cong\Dist(T_r)$.

\subsection{Cohomology of $\Uz(U_r)$} It is observed that for each $r\ge 1$, $\Uz(U_r)$ is a normal subalgebra of $\Uz(B_r)$. There is a right adjoint action of $\Uz^0$ on $\Uz(B_r)$ such that $\Uz(U_r)$ is a $\Uz^0$-submodule of $\Uz(B_r)$. Then from \cite[Theorem 4.3.1]{Dru:2011}, the cohomology $\opH^\bullet(\Uz(U_r),k)$ is a left $\Uz^0$-module. Before computing this module stucture, we need to revise some notation in Section \ref{notation}. Let $S^\bullet(x_0,\ldots,x_r)$ be the symmetric algebra over $x_0,\ldots,x_r$ of degree 2. Then this symmetric algebra can be considered as a $\Uz^0$-module by assigning weight $p^{i}\ell\alpha$ to $x_{i}$ for each $0\le i\le r$. Let $\Lambda^{\bullet}(y_0,\ldots,y_r)$ be the exterior algebra generated by $y_0,\ldots,y_r$ of degree 1. By assigning weight $\alpha$ to $y_0$ and weight $p^{i-1}\ell\alpha$ to $y_i$ for all $1\le i\le r$, we obtain the $\Uz^0$-module structure of $\Lambda^{\bullet}(y_0,\ldots,y_r)$. Now we compute the cohomology of $\Uz(U_r)$ as follows.

\begin{theorem}\label{U_rquantumcohomology}
For each $r\ge 1$, there is an isomorphism of $\Uz^0$-algebras
\[ \opH^\bullet(\Uz(U_r),k)\cong S^\bullet(x_0,\ldots,x_{r})\otimes\Lambda^\bullet(y_0,\ldots,y_{r}).\]   
\end{theorem}

\begin{proof}
Note that $\Uz(U_r)$ is in this case the same as gr$\Uz(U_r)$, the associated graded algebra of $\Uz(U_r)$. According to \cite{Dru:2011}, it is a $k$-algebra generated by $E_\alpha$ and $E_\alpha^{(p^i\ell)}$ for all $1\le i\le r-1$, subject to the relations (6.1.2), (6.1.3), and (6.1.4) in \cite{Dru:2011} applied with $\Phi^+=\{\alpha\}$. Hence, by Theorem 4.1 and Remark 4.2 in \cite{MPSW:2010}, we obtain the isomorphism as desired.
\end{proof}

\begin{remark}
This result is a special case of \cite[Proposition 6.2.2]{Dru:2011} in which Drupieski computed the cohomology of gr$\Uz(U_r)$ in a more general context.
\end{remark}

\subsection{Cohomology of $\Uz(B_r)$} Since $\Uz(B_r)//\Uz(U_r) \cong \Uz(T_r)$, for each $r \geq 1$, there exists a Lyndon-Hochschild-Serre spectral sequence
\[ E_2^{i,j}=\opH^i(\Uz(T_r), \opH^j(\Uz(U_r),k))\Rightarrow \opH^{i+j}(\Uz(B_r),k). \]
As $\Uz(T_r)$ is a semisimple Hopf algebra, the spectral sequence collapses at the second page and then we obtain the following algebra isomorphism
\[ \opH^\bullet(\Uz(B_r),k))\cong \opH^0(\Uz(T_r),\opH^{\bullet}(\Uz(U_r),k))=\opH^{\bullet}(\Uz(U_r),k)^{\Uz(T_r)}. \]
Note that the $\Uz^0$-module stucture is preserved via this isomorphism. From Theorem \ref{U_rquantumcohomology}, first taking the $\uz^0$-invariant of $\opH^{\bullet}(\Uz(U_r),k)$, we have
\begin{align*}
\opH^{\bullet}(\Uz(U_r),k)^{\uz^0}&\cong\left(\Lambda^\bullet(y_0)\right)^{\uz^0}\otimes S^\bullet(x_0,\ldots,x_r)\otimes\Lambda^\bullet(y_1,\ldots,y_r) \\
&= S^\bullet(x_0,\ldots,x_r)\otimes\Lambda^\bullet(y_1,\ldots,y_r)
\end{align*}
since $\Lambda^\bullet(y_0)^{\uz^0}=\Lambda^\bullet(y_0^\ell)=k$. Note also that each generator's weight in $\Lambda^\bullet$ is divided by $\ell$ under the Frobenius homomorphism. These results can be enclosed in computing the $\Uz(T_r)$-invariant of $\Uz(U_r)$-cohomology as follows.
\begin{align*}
\opH^{\bullet}(\Uz(U_r),k)^{\Uz(T_r)}&\cong\left( \opH^{\bullet}(\Uz(U_r),k)^{\uz^0} \right)^{\Uz(T_r)//{\uz^0}}\\
&\cong\left( S^\bullet(x_0,\ldots,x_r)\otimes\Lambda_{\zeta}^\bullet(y_1,\ldots,y_r)\right)^{\Uz(T_r)//{\uz^0}} \\
&\cong \left( S^\bullet(x'_0,x'_1,\ldots,x'_r)\otimes\Lambda^\bullet(y'_1,\ldots,y'_r) \right)^{\Dist(T_r)}
\end{align*}
where $x'_0$ is of weight $\alpha$, and each $x'_{i}$ (or $y'_{i+1}$) is of weight $p^{i}\alpha$ for all $0\le i\le r$. Now we can use an analogous argument as in Subsection \ref{B_rcohomology} to compute the $\Uz(B_r)$-cohomology.
  
\begin{theorem}\label{B_r-quantumcohomology}
For each $r\ge 1$, there is an isomorphism of $T$-modules
\[ \opH^\bullet(\Uz(B_r),k)^{(-r)}\cong\bigoplus_{n\in\mathbb{N}}(n\alpha)^{N'_r(p,n)} \]
where $N'_r(p,n)$ is the number of solutions $(a_0,\ldots,a_r,b_1,\ldots,b_r)\in\mathbb{N}^{r+1}\times\{0,1\}^r$ for the equation
\[ (a_0+b_1)+p(a_1+b_2)+\cdots+p^{r-1}(a_{r-1}+b_r)+p^ra_r=\frac{n}{2}p^{r}.\] 
\end{theorem}

\begin{remark}
Comparing with the equation \eqref{lambda-eqn} in Section \ref{cohomology}, we have 
\[ N'_r(p,n)=\sum_{m=0}^\infty N_r(p,m,n). \]
It is also observed that $N_r(p,m,n)=0$ when $m>np^r$ so the sum above is well-defined and $N'_r(p,n)$ is finite. 
\end{remark}

\subsection{Cohomology of $\Uz(G_r)$} We first need to develop a spectral sequence which is similar to the one in Subsection \ref{G_rspectralsequence}.

\begin{lemma}
Let $M$ be a $\Uz(B)$-module such that $R^i\ind_{\Uz(B)}^{\Uz}M=0$ for each $i\ge 1$. Then there is for each $r\ge 1$ a spectral sequence of $G$-modules
\[ R^i\ind_{B}^{G}\opH^j(\Uz(B_r),M)^{(-r)}\Rightarrow\left( \opH^{i+j}(\Uz(G_r),\ind_{\Uz(B)}^{\Uz}M)\right)^{(-r)}. \]
\end{lemma}

\begin{proof}
Note that the category of rational $B$-modules (resp. $G$-modules) is equivalent to the category of integrable $\Dist(B)$-modules (resp. locally finite $\Dist(G)$-modules) \cite[6.9, 9.4]{CPS:1980}. Identifying $\Dist(H)$ with $\Uz(H)//\uz$ via the restriction of Frobenius homomorphism $F_{\zeta}$ where $H$ is $G, B, G_r,$ or $B_r$. Let 
\begin{align*}
\calF_1(-) &=\ind_B^G(-)\circ(-)^{\Uz(B_r)(-r)},\\
\calF_2(-)&=(-)^{\Uz(G_r)(-r)}\circ\ind_{\Uz(B)}^{\Uz}(-)
\end{align*}
being functors from the category of $\Uz(B)$-modules to the category of rational $G$-modules. From \cite[Proposition I.4.1]{Jan:2003}, we have the following spectral sequences
\begin{align}
R^i\ind_{B}^{G}\opH^j(\Uz(B_r),M)^{(-r)} &\Rightarrow R^{i+j}\calF_1(M),\label{spectral 1}\\
\opH^{i}(\Uz(G_r),R^j\ind_{\Uz(B)}^{\Uz}M)^{(-r)} &\Rightarrow R^{i+j}\calF_2(M).\label{spectral 2}
\end{align}
On the other hand, we consider
\[ \ind_B^G(M^{\Uz(B_r)})^{(-r)}\cong \ind_B^G(M^{\uz(B)})^{\Uz(B_r)//\uz(B)(-r)}\cong \ind_B^G(M^{\uz(B)})^{B_r(-r)}\]
which is isomorphic to the functor $\left(\ind_B^GM^{\uz(B)}\right)^{G_r(-r)}$ by \cite[3.1]{AJ:1984}. In addition, we have
\begin{align*}
\left(\ind_B^GM^{\uz(B)}\right)^{G_r} &\cong \left(\ind_B^GM^{\uz(B)}\right)^{\Uz(G_r)//\uz} \\
&\cong \left[ \left(\ind_{\Uz(B)}^{\Uz}M\right)^{\uz}\right]^{\Uz(G_r)//\uz}\\
&\cong \left(\ind_{\Uz(B)}^{\Uz}M\right)^{\Uz(G_r)}
\end{align*}
where the second isomorphism is because of Theorem 5.1.1 in \cite{Dru:2011}. In summary, the two functors $\calF_1$ and $\calF_2$ are naturally isomorphic. Hence, the spectral sequences \eqref{spectral 1}, \eqref{spectral 2} converge to the same abutment. Since $R^i\ind_{\Uz(B)}^{\Uz}M=0$ for each $i\ge 1$, the second spectral sequence collapses and then gives us the desired spectral sequence.
\end{proof}

\begin{remark}\label{cup-product}
As the isomorphism between $\calF_1$ and $\calF_2$ respects the cup-products, similar argument as in \cite[Remark 3.2]{AJ:1984} shows that the spectral sequence in the lemma above is compatible with the $\opH^\bullet(\Uz(G_r),k)$-module structure.
\end{remark}

Now we can compute the cohomology of $\Uz(G_r)$.

\begin{theorem}
There are for each $r\ge 1$ $G$-module isomorphisms
\[ \opH^{\bullet}(\Uz(G_r),k)^{(-r)}\cong\ind_B^G\opH^\bullet(\Uz(B_r),k)^{(-r)}\cong\bigoplus_{n\in\mathbb{N}}\ind_B^G(n\alpha)^{N'_r(p,n)}\]
in which the first isomorphism is of graded $k$-algebras.
\end{theorem}
\begin{proof}
Note first that $R^i\ind_{\Uz(B)}^{\Uz} k=0$ for all $i\ge 1$ by the quantum version of Kempf's Vanishing Theorem \cite[Theorem 5.5]{RH:2003}. The preceding lemma implies the following spectral sequence
\[ R^i\ind_B^G\opH^j(\Uz(B_r),k)^{(-r)}\Rightarrow \opH^{i+j}(\Uz(G_r),\ind_{\Uz(B)}^{\Uz} k)^{(-r)}. \]

On the other hand, Theorem \ref{B_r-quantumcohomology} shows that $\opH^j(\Uz(B_r),k)^{(-r)}$ is decomposed into dominant weight spaces for each $j\ge 0$. Hence, we have
\[ R^i\ind_B^G\opH^j(\Uz(B_r),k)^{(-r)}=0\]
for all $i>0$, which implies the collapse of the spectral sequence so that we have for each $j\ge 0$
\[ \ind_B^G\opH^j(\Uz(B_r),k)^{(-r)}\cong \opH^{j}(\Uz(G_r),k)^{(-r)}\]
as a $G$-module. By Remark \ref{cup-product}, these isomorphisms extend to the first isomorphism of the theorem. The second one immediately follows from Theorem \ref{B_r-quantumcohomology}.
\end{proof}

\subsection{Cohen-Macaulay quantum cohomology} We continue establishing in this section analogous results as in Sections \ref{reduced cohomology} and \ref{Cohen-Macaulayrings}. As most of the arguments here are similar to those of the aforementioned sections, we omit details of proofs below.

We begin by looking at the reduced parts of quantum cohomology rings.

\begin{proposition}
For each $r\ge 0$, there are isomorphisms of rings
\begin{align*}
\opH^\bullet(\Uz(B_r),k)_{\red} &\cong S^\bullet(x_0,\ldots,x_r)^{T_r}, \\
 \opH^\bullet(\Uz(G_r),k)_{\red} &\cong \ind_B^GS^\bullet(x_0,\ldots,x_r)^{T_r}.
\end{align*}
\end{proposition}

\begin{proof}
We have 
\begin{align*}
\opH^\bullet(\Uz(B_r),k)_{\red} &\cong\left( \opH^\bullet(\Uz(U_r),k)^{\Uz(T_r)}\right)_{\red}\\
&\cong \left( (\opH^\bullet(\Uz(U_r),k)^{\uz^0})^{\Uz(T_r)//{\uz^0}}\right)_{\red}\\
&\cong \left( [S^\bullet(x_0,\ldots,x_r)\otimes\Lambda^\bullet(y_1,\ldots,y_r)]^{T_r}\right)_{\red}\\
&\cong \left( [S^\bullet(x_0,\ldots,x_r)\otimes\Lambda^\bullet(y_1,\ldots,y_r)]_{\red}\right)^{T_r}\\
&\cong S^\bullet(x_0,\ldots,x_r)^{T_r}
\end{align*}
where the fourth isomorphism is from Lemma \ref{lemma of reduceB_r}. Next, Lemma \ref{lemma of reduceG_r} gives us the second isomorphism of the proposition.
\end{proof}

Now the following theorem shows that the Frobenius-Lusztig kernels of quantum groups also have Cohen-Macaulay cohomology.

\begin{theorem}
For each $r\ge 0$, the cohomology rings $\opH^\bullet(\Uz(B_r),k)_{\red}$ and $\opH^\bullet(\Uz(G_r),k)_{\red}$ are Cohen-Macaulay.
\end{theorem}

\begin{proof}
The argument is similar to that given for Theorem \ref{Cohen-Macaulay for B_r and G_r}.  
\end{proof}

\section*{Acknowledgments}

This paper is developed from a part of the author's Ph.D. thesis. The author gratefully acknowledges the guidance of his thesis advisor Daniel K. Nakano. We deeply thank Christopher M. Drupieski, who spent a lot of time and energy to read and correct our preprints. We are also grateful the conversations with David Krumm. We thank Wilberd van der Kallen and the referee for useful comments and suggestions on the manuscript. Last but not least, the author would like to thank Jon F. Carlson for teaching him MAGMA programming and giving him an account on the SLOTH machine to perform all the calculations.

\section{Appendix}\label{appendix}
\subsection{} In this section we compare the effectiveness of two programs computing the number $N_r(p,m,n)$. The first program is encoded from the algorithm in Section \ref{character multiplicity} and denote by $F_1$ in the tables below. The other one is applied Ehrhart's theory on counting integral lattice points in a polytope and denote by $F_2$.\footnote{Both codes are available on arxiv.org.} Note also that both programs are written in MAGMA code and run on the MAGMA computer algebra system.

In both tables, we fix $m=0$ and respectively consider $p=3$ in Table 1, and $p=5$ in Table 2. The last two columns show the time needed to compute $N_r(p,m,0)$ for all $0\le m\le 10$. The results show that $N_r(p,0,n)=0$ if $n$ is odd, hence only results for $m$ even are exhibited in the tables.
  
\begin{table}[ht]
\caption{$p=3$} 
\centering 
\begin{tabular}{| c | c c c c c c | c | c | } 
\hline
$r\backslash n$ & 0 & 2 & 4 & 6 & 8 & 10  & $F_1$  & $F_2$  \\ [0.5ex] 
\hline 
 2 & 1 & 3 & 5 & 7 & 9 & 11  & 0  & 1.56 (s)  \\  [1ex] 
\hline 
3 & 1 & 13 & 37 & 73 & 121 & 181  &  0  & 5.74 (s)  \\  [1ex] 
\hline
4 & 1 & 111 & 545 & 1519 & 3249 & 5951  & 0.01 (s)  & 41.04 (s)  \\  [1ex]
\hline
5 & 1 & 2065 & 17857 & 70705 & 195601 & 439201  & 0.77 (s)  & 67392.07 (s)  \\  [1ex]
\hline
\end{tabular}
\label{table:nonlin} 
\end{table}
\begin{table}[ht]
\caption{$p=5$} 
\centering 
\begin{tabular}{| c | c c c c c c | c | c | } 
\hline
 $r\backslash n$ & 0 & 2 & 4 & 6 & 8 & 10  & $F_1$  & $F_2$  \\ [0.5ex] 
\hline 
 2 & 1 & 3 & 5 & 7 & 9 & 11  & 0  & 1.86 (s)  \\  [1ex] 
\hline 
3 & 1 & 21 & 61 & 121 & 201 & 301  &  0  & 7.15 (s)  \\  [1ex] 
\hline
4 & 1 & 503 & 2505 & 7007 & 15009 & 27511  & 0.08 (s)  & 43.86 (s)  \\  [1ex]
\hline
5 & 1 & 42521 & 377561 & 1505121 & 4175201 & 9387801  & 16.78 (s)  & 69773.82 (s)  \\  [1ex]
\hline
\end{tabular}
\label{table:nonlin} 
\end{table}



\providecommand{\bysame}{\leavevmode\hbox to3em{\hrulefill}\thinspace}

\end{document}